\documentclass[12pt,a4paper]{article}
\usepackage[english]{babel}
\usepackage[utf8]{inputenc}
\usepackage{amsmath,amssymb,amsthm,graphicx}
\usepackage{geometry}
\usepackage{natbib}
\usepackage[all]{xy}
\usepackage{tikz}
\usepackage{url}

\title{Cholesky decompositions of integral operators and the Fredholm determinant}

\author{Niels Lundtorp Olsen}
\date{}

\newcommand{\vektor}[1]{\begin{pmatrix} #1 \end{pmatrix}}

\newcommand{\pil}{\rightarrow}

\newcommand{\M}[1]{\mathbb{#1}}
\newcommand{\R}{\mathbb{R}}

\newcommand{\brak}[1]{\langle #1 \rangle}

\newcommand{\trans}{^\top \!}

\newcommand{\I}{\mathbf{I}}

\newcommand{\de}{\: \mathrm{d}}
\newcommand{\inte}[4]{\int_{#1}^{#2} \! #3 \de {#4} }

\theoremstyle{definition}
\newtheorem{cond}{Condition}

\newtheorem{lemma}[cond]{Lemma}

\newtheorem{thm}{Theorem}

\theoremstyle{definition}

\begin{document}

\maketitle

This paper presents an operator version of the well-known and popular tool for calculating the determinant of matrix using its Cholesky decomposition:
$$\log\det M = 2\sum_{i=1}^{n} \log L_{ii}$$
where $LL\trans = M$ is the Cholesky decomposition of $M$. 

The equivalent result for operators in $L^2[0,1]$ is
$$\log \det (\I + M) := \log (\text{Fredholm determinant of }M) = \inte{0}{1}{T(t,t)}{t}.$$
where $M$ is an integral operator, and $(\I + T)(\I + T^*) = (\I + M)$ is the Cholesky decomposition of $M$. 

\subsection*{Background}

Consider the Hilbert space $\mathbb{H}$ of square-integrable functions on $[0,1]$, $\mathbb{H} = \mathcal{B}(L^2[0,1]; \R)$. 
Let $K$ be an integral operator in $\mathcal{B}(\M{H})$. % with kernel $K(x_p, x_q) $.
We shall use the same symbol for an integral operator and its kernel, ie.:
$$
K f(t) = \inte{0}{1}{K(t,s) f(s)}{s}, \quad f \in L^2[0,1]
$$
Furthermore, we shall denote the identify operator by  $\I$ and use $^\ast$ to denote the adjoint of an operator. 

We define the Fredholm determinant for an integral operator $K \in \mathcal{B}(L^2[0,1]; \R)$ by:
\begin{equation}
Fd(K) = \sum_{k=0}^\infty \frac{1}{n!} \int_0^1 \dots \int_0^1 \det [K(x_p, x_q) ]_{p,q = 1}^n \de x_1 \dots \de x_n,
\end{equation}
Here the function $z \mapsto Fd(zK)$ is an entire function on the complex plane \citep{fredholm1903}.
$Fd(K)$ is well-defined if $K$ is block-continuous, but \emph{not} for triangular operators (as the integral along the diagonal is ambiguous if $K$ is discontinuous across the diagonal).
We refer to the discussion in \cite{bornemann2010} regarding different definitions of the Fredholm determinant. \medskip

The Fredholm determinant is the equivalent of the "usual" determinant for operators on $\R^n$ (ie. matrices) and was introduced in \cite{fredholm1903} to solve integral equations on the form $\phi = f + zK\phi$ for unknown $\phi \in L^2[0,1]$.  

\subsection*{Cholesky decomposition for matrices}
One of the most appreciated computational tools for positive definite matrices is the \emph{Cholesky decomposition}: Let $M \in \R^{n\times n}$ be positive definite. Then there exists a unique decomposition $M = LL\trans$, where $L$ is a lower triangular matrix. There exist specialised algorithms for computing the Cholesky decomposition, something which many programming languages implements; the Wikipedia entry on the Cholesky decomposition \footnote{\url{https://en.wikipedia.org/wiki/Cholesky_decomposition}} has a good overview. 
%fast and numerically stable algorithms for determining this decomposition.

The Cholesky decomposition may aid in various computations involving $M$. Assume $M$ has the Cholesky decomposition $ M = LL\trans$ and assume $y \in \R^n$: %. We then have:
\begin{itemize}
	\item Inner product using $M$ or $M^{-1}$: $y\trans M y = ||L\trans y||^2$, $y\trans M^{-1} y = ||L^{-1} y||^2$.
	\item Inversion of $M$: $M^{-1} = L^{-1, \top}L^{-1}$.
	\item Log-determinant of $M$: $\log\det M = 2\sum_{i=1}^{n} \log L_{ii}$
\end{itemize}
where the triangular structure furthermore can be used for fast calculations of $L^{-1}$ and $L^{-1}y$. 

\subsection*{Cholesky decomposition for operators on $L^2[0,1]$}
We can also define Cholesky decompostions for integral operators on $\M{H}$, though some care needs to be taken as integral operators are not generally invertible. 

Let $A \in \mathcal{B}(\M{H})$ be a positive definite integral operator. We define the cholesky decomposition of $\I + A$ as $(\I + T) (\I + T^*)$, where $T$ is a triangular operator (ie. an integral operator where $T(x,y) = 0$ for $x > y$).  % TJEK FORTEGN PÅ ULIGHED

The Cholesky decomposition of $\I + A$ satisfy inversion and inner product properties similar to matrices:
$$(I+A)^{-1} = (\I + T^{-1,\ast})(\I + T^{-1}) $$
and
$$\brak{f,(I+A)^{-1}f} = ||(\I + T^{-1})f||^2 $$
We refer to \cite[chap. XXII]{gohberg1993} for a general discussion on LU decompositions of operators on $L^2[0,1]$.

\paragraph{The Cholesky decomposition and the determinant} %
The determinant is not well-defined for operators on $L^2[0,1]$, so it is not obvious if or how one could have a statement similar to $\log\det M = \sum_{i=1}^{n} \log L_{ii}$. % from above. 
However, knowing that the Fredholm determinant plays a role with many similarities to the ordinary determinant, we could hope for a statement along these lines.

This is indeed the case if we define $\det (\I + M) =  Fd(M)  $. The theorem is the following: 
%We let $\I$ refer to the identity operator and $^*$ to the adjoint. 
%
\begin{thm} \label{fdet-dekomp-thm}
Let $A$ be an integral operator on $[0,1]$ with continuous kernel. Suppose $\I+A$ decomposes $\I + A = (\I + T) (\I + T^*)$ where $T$ is a triangular operator. 

Assume that $T$ is continuous on the closed lower triangle $\{(x,y) | 0 \leq x \leq 1, 0 \leq y \leq x \}$ (in general, $T$ is not continuous as a function on $[0,1]^2$). We then have
$$\log \det (\I + A) = \inte{0}{1}{T(t,t)}{t}.$$

\end{thm}
The theorem can  be justified heuristically from the 'matrix-like' approximation of $\I + A$: $1 + T(t,t)$ corresponds to the diagonal elements, and $\log(1 + T(t,t)) \approx T(t,t)$ when $T(t,t)$ is small. 

\begin{proof}
In the following, we define  $\det (\I + M) =  Fd(M)  $ whenever $Fd(M)$ is well-defined. 
	
	It can easily be proved that the Fredholm determinant has the block-property of ordinary matrices; ie.:
	$$
	\det \vektor{\I + U & 0 \\ V & \I + X} = \det \vektor{\I + U & V \\ 0 & \I + X} = \det(\I + U) \det(\I + X)
	$$
	for continuous integral operators $U, V, X$ of appropriate sizes.
	
	Define $A_t$ and $T_t$ to be the restrictions of $A$ and $T$ to $[0,t]$ for $0 \leq t < 1$. By construction, $\I+A_t = (\I + T_t) (\I + T_t^*)$.
	
	Fix $t$, and let $ s > 0$ s.t. $t + s < 1$. %Define $X_s$ to be the restriction of $M$ to $[t, t+s]$. 
	Suppose $\I + T_{t+s}$ and $\I + A_{t+s}$ admit the following decompositions:
$$
	\I + T_{t+s} = \vektor{\I + T_t & 0 \\ K_s & \I + X_s}
$$
and
$$
	\I + A_{t+s} = \vektor{\I + A_t & B_s^* \\ B_s & \I + C_s}$$
	Using the Schur complement, we get
	$$
	\vektor{\I + A_t & B_s^* \\ B_s & \I + C_s} = \vektor{\I & 0 \\ B_s(1+A_t)^{-1} & \I} \vektor{\I + A_t &  B_s^* \\ 0  & \I + C_s - B_s(I +A_t)^{-1} B_s^*}
	$$
	which is valid since $\I + A_{t+s}$ is positive definite. By expanding the relation $(\I + T_{t+s})(\I + T_{t+s}) = \I + A_{t+s}$, one can show that 
	 that $(\I+ X_s)(\I+ X_s^*) = \I + C_s - B_s(\I + A_t)^{-1} B_s^*$, and therefore
	\begin{equation}\label{a-dekomp-x}
	\det \vektor{\I + A_t & B_s^* \\ B_s & \I + C_s} = \det \vektor{\I + A_t}  \det \vektor{\I + X_s + X_s^* + X_sX_s^*}.
	\end{equation}
	
	\begin{lemma} \label{x-graense}
		Assume $T$ is continuous in a neighbourhood around $(t,t)$ as a function on the closed lower triangle. 
		Then
		\begin{equation}
		\lim_{s \pil 0}  \frac{\det (I + X_s + X_s^* + X_sX_s^*) - 1}{s} = M(t,t). \label{lemma-a1-eq}
		\end{equation}
		ie. 
		$$
		\frac{d}{ds} \det (I + X_s + X_s^* + X_sX_s^*)\Big|_{s = 0} = M(t,t)
		$$

	\end{lemma}
	
	\begin{proof}
		Define $Y_s := X_s + X_s^* + X_sX_s^*$.
		
		Since $T$ is continuous, $X_s$ is bounded and hence the kernel for $X_sX_s^*$ decreases rapidly towards zero for $ s \pil 0$. Using continuity of $T$ in $(t,t)$, we conclude that the kernel of $Y_s$ deviate from $T(t,t)$ by less than $\epsilon_s$ where $\epsilon_s \pil 0$ for $s \pil 0$. 
		
		Using dominated convergence on the terms of the Fredholm determinant (indexed by $n$ starting from zero), we will show that all higher-order terms of eq. \eqref{lemma-a1-eq} vanish in the limit, leaving an expression resembling a first-order Taylor expansion. 
		
		Let $n \geq 2$. The corresponding entry in the Fredholm determinant of $Y_s$ is: 	
		\begin{equation*}
		\frac{1}{n!} \int_{[0,s]^n} \det[(Y_s)(x_p,x_q)]_{p,q = 1}^n \de x_1 \dots \de x_n 
		< \\
		s^{n} (T(t,t) + \epsilon_s)^n 
		\end{equation*}
		which tend to zero for $s \pil 0$ when multiplied by $s^{-1}$.
		The entry for $n = 1$ is
		$$
		 \inte{0}{s} {Y_s(x,x)} {x}
		$$
		which multiplied by $s^{-1}$ converges to $T(t,t)$ for $s \pil 0$. The entry for $n = 0$ is trivially 1. 
		
Now select $\tilde{s} > 0 $ such that $\epsilon_{\tilde{s}} < 1$ and $\tilde{s} (T(t,t) + 1) < \tfrac{1}{2}$.
Then 
\begin{multline*}
s^{-1} (d(Y_s) - 1) < s^{-1}\inte{0}{s} {Y_s(x,x)} {x} + \sum_{n = 2}^\infty s^{n-1} (T(t,t) + \epsilon_s)^n < \\
 s^{-1}\inte{0}{s} {(T(t,t) + 1)} {x}  + \sum_{n = 2}^\infty \tilde{s}^{n-1} (T(t,t) + 1)^n < \\
 {(T(t,t) + 1)}  +  (T(t,t) + 1)  \sum_{n = 2}^\infty (1/2)^{n-1}
\end{multline*}	
for $0 < s < \tilde{s}$. 
Hence we may apply the dominated convergence theorem and conclude
\begin{multline*}
\lim_{s \pil 0} s^{-1} (Fd(Y_s) - 1)  = \lim_{s \pil 0} \left( -s^{-1} +  \sum_{n = 0}^\infty \frac{1}{s(n!)} \int_{[0,s]^n} \det[(Y_s)(x_p,x_q)]_{p,q = 1}^n \de x_1 \dots \de x_n \right) = \\
 \sum_{n = 1}^\infty \lim_{s \pil 0}  \frac{1}{s(n!)} \int_{[0,s]^n} \det[(Y_s)(x_p,x_q)]_{p,q = 1}^n \de x_1 \dots \de x_n = T(t,t)
\end{multline*}	
		This proves the result. 
	\end{proof}
	
	Now, returning to $A_{t+s}$,
	\begin{multline*}
	\frac{d}{dt} \det (\I + A_t) = \lim_{s \pil 0}  \frac{\det (\I + A_{t+s}) -  \det (\I + A_t) }{s} = \\
	\det (\I + A_t)  \lim_{s \pil 0} \frac{\det (\I + X_s + X_s^* + X_sX_s^*)}{s} = \det (\I + A_t) M(t,t)
	\end{multline*}
	using Eq \eqref{a-dekomp-x} and Lemma \ref{x-graense}. Therefore,
	\begin{equation}
	\frac{d}{dt} \log \det (\I + A_t) = T(t,t) \label{ddtlogmt}
	\end{equation}
	Using that $\det (\I + A_0) = 1$ and integrating \eqref{ddtlogmt}, we get 
	$$
	\log \det (\I + A_t) = \inte{0}{t}{T(s,s)}{s}, \quad 0 \leq t \leq 1
	$$
\end{proof}

\subsection*{Discussion}
The Fredholm determinant has a definition (Eq. 1) which makes it very hard to evaluate by direct calculation. This work presents a  simple and elegant solution to this problem using the Cholesky decomposition.
Naturally, applications of Theorem \ref{fdet-dekomp-thm} rely on knowing the Cholesky decomposition of the operator $\I + A$. 
One approach is to do applications with  triangular matrices as the building block: we design or propose a triangular operator $T$ (or a class of triangular operators), and then construct the positive definite matrix $\I + A = (\I + T)(\I + T^*)$. 

\subsection*{Acknowledgements}
I am grateful to Associate Professor Bo Markussen (Univeristy of Copenhagen) and Professor Folkmar  Bornemann (TU München) for comments to the theorem and proof. In particular I would like to thank Professor Bornemann for pointing out an alternative proof using the theory of resolvent kernels.

\bibliographystyle{agsm}

\end{document}